\documentclass[11pt,reqno]{amsart}
\usepackage{a4wide}
\usepackage{amsmath,amssymb,amsthm,mathtools}
\usepackage[mathcal]{eucal}
\usepackage{mathrsfs}
\usepackage{enumitem}
\usepackage[colorlinks=true,linkcolor=blue,citecolor=blue]{hyperref}

\newtheorem{theorem}{Theorem}[section]
\newtheorem{lemma}[theorem]{Lemma}
\newtheorem{proposition}[theorem]{Proposition}
\newtheorem{corollary}[theorem]{Corollary}
\newtheorem{question}[theorem]{Question}
\theoremstyle{definition}

\newtheorem{convention}[theorem]{Convention}
\theoremstyle{remark}
\newtheorem{remark}[theorem]{Remark}

\newcommand{\ZF}{\mathsf{ZF}}
\newcommand{\ZFC}{\mathsf{ZFC}}
\newcommand{\Sym}{\operatorname{Sym}}
\newcommand{\End}{\operatorname{End}}

\newcommand{\At}{\operatorname{At}}

\newcommand{\Ext}{\operatorname{Ext}}
\newcommand{\Ball}{\operatorname{Ball}}

\newcommand{\Rel}{\operatorname{Rel}}
\newcommand{\cP}{\mathcal{P}}
\newcommand{\cK}{\mathcal{K}}
\newcommand{\cC}{\mathcal{C}}
\newcommand{\cB}{\mathcal{B}}

\newcommand{\cT}{\mathcal{T}}
\newcommand{\bbT}{\mathbb{T}}
\newcommand{\bbF}{\mathbb{F}}
\newcommand{\bbK}{\mathbb{K}}

\newcommand{\bbQ}{\mathbb{Q}}
\newcommand{\bbR}{\mathbb{R}}
\newcommand{\bbC}{\mathbb{C}}
\newcommand{\bbZ}{\mathbb{Z}}

\title[Canonical reconstruction and forcing absoluteness]%
{Canonical reconstruction and forcing absoluteness \\ of standard structures}

\author[T. Kania]{Tomasz Kania}
\address[T. Kania]{Mathematical Institute,
Czech Academy of Sciences,
\v{Z}itn\'a 25,
115 67 Praha 1,
Czech Republic
\and
Institute of Mathematics and Computer Science,
Jagiellonian University,
\L{}ojasiewicza 6,
30-348 Krak\'ow,
Poland}
\email{kania@math.cas.cz, tomasz.marcin.kania@gmail.com}

\thanks{IM CAS (RVO 67985840).}

\date{\today}



\begin{document}

\begin{abstract}
We isolate a simple preservation principle governing when it is absolute,
between transitive models of set theory, that a given algebraic or
topological-algebraic structure has a \emph{standard form} $F(X)$ indexed
by a set~$X$. The principle is: if the index~$X$ (or a proxy for it) can
be recovered from~$F(X)$ by a uniform definable construction, then the
class of structures isomorphic to some~$F(X)$ is downward absolute from
forcing extensions. Answering a question raised by Noah Schweber, we
deduce in particular that no group that fails to be a full symmetric
group in the ground model can become one after forcing; the result holds
already in~$\ZF$. The same mechanism applies to full transformation
monoids, powerset Boolean algebras, full relation algebras, full clones,
full partition lattices, products $R^X$ of finitely generated centrally
indecomposable rings, the commutative $C^*$-algebras $\ell_\infty(X)$ and
$c_0(X)$, full endomorphism rings, the operator algebras $\cB(H)$ and
$\cK(H)$, and $\ell_1(X)$ as a real Banach lattice. In the motivating
symmetric-group case, the same reconstruction gives more than descent:
it yields a uniform $\Pi^1_1$ definition of fullness over transitive
$\ZF$-models. We then exhibit clean torsor obstructions, in the
standard symmetric-model situation: \emph{finite covers}
$Y \times n$ already separate $\ZF$-failure from $\ZFC$-descent without
any completeness caveat, and the finite-support normed space $c_{00}(I)$
provides the analogous Banach example. Bare-Banach-space isomorphism
with $\ell_1(\Gamma)$ exhibits a genuine $\ZFC$-descent. We conclude
with the corresponding, relative, obstructions to
$\Pi^1_1$-definability of standardness over transitive $\ZF$-models.
\end{abstract}

\maketitle

\section{Introduction}

A recurring question, which surfaces in many guises across algebra and
analysis, is whether being a ``standard'' object of a given type is
itself a definable property of the abstract structure. One instance was
raised by Noah Schweber: if a group $G$ fails to be isomorphic to any
full symmetric group $\Sym(X)$, can forcing make it one? More generally,
for a class $\cC = \{F(X) : X \text{ a set}\}$ of \emph{standard
forms} of some structure type, is membership in $\cC$ preserved as we
move between transitive models of $\ZF$ or $\ZFC$?

On its face the predicate ``$A \cong F(X)$ for some $X$'' is
$\Sigma^1_2$: it asserts the existence of an index~$X$ and of an
isomorphism $A \to F(X)$. More explicitly, in the examples considered
here the elements of $F(X)$ are typically subsets of~$X$, functions
on~$X$, relations on~$X$, operators on a space indexed by~$X$, or
families indexed by~$X$. Thus the assertion that a proposed map
$A \to F(X)$ is onto contains a universal quantifier over such objects,
not merely over the elements of~$X$. With $X$ as the first second-order
witness, this is the additional projective quantifier responsible for
the $\Sigma^1_2$ complexity. The familiar ``potential counterexample''
heuristic for ruling out a $\Pi^1_1$ definition of such a class is to
find $A$ outside $\cC$ in a ground model and inside $\cC$ in a forcing
extension. The goal of this note is to record a uniform obstruction to
this heuristic in many natural cases, and to delineate the cases where
the heuristic really does succeed.

For full symmetric groups the situation is stronger still. The
canonical reconstruction of the point set from the abstract group
yields a direct $\Pi^1_1$ definition of fullness. Thus, in the
motivating example, the forcing strategy cannot work not merely because
fullness is downward absolute, but because fullness is already
$\Pi^1_1$ over transitive $\ZF$-models.

The key observation is elementary but unifying. Call a standard form
$F(X)$ \emph{canonically reconstructible} if the index set~$X$ can be
recovered, up to a definable bijection, from the abstract structure of
$F(X)$ by a formula of the appropriate signature. Whenever this holds,
isomorphism with some $F(X)$ is downward absolute between any two
transitive models of $\ZF$, one of which contains the other:

\begin{quote}
\emph{If the index of a standard structure $F(X)$ can be reconstructed
canonically from its abstract structure, then the predicate
``$A \cong F(X)$ for some $X$'' is downward absolute from forcing
extensions, and indeed from any outer transitive $\ZF$-model.}
\end{quote}

The first main result illustrates the programme.

\begin{theorem}[Full symmetric groups; $\ZF$]\label{thm:intro-sym}
Let $M \subseteq N$ be transitive models of $\ZF$ and let $G \in M$ be
a group. If $N \models$ ``$G \cong \Sym(X)$ for some set $X$'', then
$M \models$ ``$G \cong \Sym(Y)$ for some set $Y$''.
\end{theorem}

In particular, no forcing extension of any $\ZFC$ model can turn a
non-full group into a full symmetric group. The same scheme of
reconstruction applies to a variety of structures. We shall prove
unconditional $\ZF$-descent for:
\begin{itemize}[nosep,label=$\bullet$]
  \item full symmetric groups and full transformation monoids;
  \item powerset Boolean algebras and Boolean rings of all subsets;
  \item full relation algebras $\Rel(X) = \cP(X \times X)$;
  \item full clones $\mathscr{O}_X$ of all finitary operations on~$X$;
  \item full partition lattices $\Pi(X)$;
  \item products $R^X$ where $R$ is finitely generated as a unital ring
    and has no central idempotents other than $0$ and $1$ — covering
    $\bbZ$, all $\bbZ/p^m\bbZ$, all finite fields $\bbF_{p^n}$, all
    matrix rings $M_k(\bbF_q)$, and all polynomial rings
    $\bbZ[t_1, \ldots, t_n]$;
  \item the commutative $C^*$-algebras $\ell_\infty(X)$ and $c_0(X)$;
  \item full endomorphism rings $\End_D(V)$ of non-zero vector spaces
    possessing a rank-one complemented line;
  \item the operator algebras $\cB(H)$ and $\cK(H)$ on a non-zero Hilbert
    space;
  \item $\ell_1(X)$ as a real Banach lattice.
\end{itemize}

For each of the above, the index can be reconstructed canonically by an
absolute formula. We shall also prove $\ZFC$-descent for several
choice-dependent examples. For finite covers and $c_{00}(\Gamma)$ we
give explicit $\ZF$-obstructions; for Hilbert spaces the obstruction is
the familiar basis-existence issue:
\begin{itemize}[nosep,label=$\bullet$]
  \item finite covers — the class of equivalence relations isomorphic
    to $Y \times n$ for some set $Y$ and a fixed integer $n \geqslant 2$;
  \item bare normed-space isomorphism with $c_{00}(\Gamma)$;
  \item bare Banach-space isomorphism with $\ell_1(\Gamma)$;
  \item Hilbert-space isomorphism with $\ell_2(\Gamma)$ — a
    basis-existence issue.
\end{itemize}

The $\ZF$-obstructions are always of the same form: a bundle of locally
isomorphic factors without a global trivialisation. Forcing, through
the generic production of a section, can turn the bundle into a trivial
standard object.

The following table summarises the catalogue; it should be read as a
guide to the theorems and propositions proved in the body of the paper.

\medskip\noindent
\begin{tabular}{@{}p{0.48\textwidth}ll@{}}
\textbf{Standardness class} & \textbf{Descent} & \textbf{Reference}\\[2pt]
\hline\\[-8pt]
$G \cong \Sym(X)$ & $\ZF$; in fact $\Pi^1_1$ &
  Thm~\ref{thm:symmetric-groups}, Cor~\ref{cor:sym-pi11}\\
$S \cong X^X$ & $\ZF$ & Thm~\ref{thm:transformation-monoid}\\
$B \cong \cP(X)$ & $\ZF$ & Thm~\ref{thm:powersets}\\
$A \cong \Rel(X)$ & $\ZF$ & Thm~\ref{thm:relation-algebras}\\
$C \cong \mathscr{O}_X$ (full clones) & $\ZF$ & Thm~\ref{thm:full-clones}\\
$L \cong \Pi(X)$ (partition lattices) & $\ZF$ & Thm~\ref{thm:partition-lattices}\\
$A \cong R^X$, $R$ f.g.\ centrally indecomp.\ & $\ZF$ & Thm~\ref{thm:fg-products}\\
$A \cong \ell_\infty(X),\ c_0(X)$ & $\ZF$ & Thm~\ref{thm:cstar-commutative}\\
$A \cong \End_D(V)$, $D$ varying & $\ZF$ & Thm~\ref{thm:endomorphism-rings}\\
$A \cong \cB(H),\ \cK(H)$ & $\ZF$ & Thms~\ref{thm:BH},~\ref{thm:KH}\\
$E \cong \ell_1(X)$ as real Banach lattice & $\ZF$ & Thm~\ref{thm:l1-lattice}\\[3pt]
\hline\\[-8pt]
$(E, \sim) \cong Y \times n$ (finite covers) & $\ZFC$ (strict) &
  Props~\ref{prop:finite-covers-zfc},~\ref{prop:finite-covers-zf-fail}\\
$E \cong c_{00}(I)$ as real normed space & $\ZFC$ (strict) &
  Thm~\ref{thm:c00-zfc},\ Prop~\ref{prop:c00-sign-torsor}\\
$E \cong \ell_1(\Gamma)$ as Banach space & $\ZFC$ &
  Thm~\ref{thm:l1-zfc},\ Prop~\ref{prop:l1-zf-failure}\\
$H \cong \ell_2(\Gamma)$ & $\ZFC$; $\ZF$ basis issue &
  Prop~\ref{prop:hilbert-zfc}\\
\end{tabular}
\medskip

The paper is organised as follows. Section~\ref{sec:descent} isolates
the descent lemma and its reformulation in terms of
$\Pi^1_1$-absoluteness. Section~\ref{sec:zf} treats the $\ZF$-descent
examples. Section~\ref{sec:zfc} treats the $\ZFC$-descent examples and
presents the torsor-style counterexamples. Section~\ref{sec:pi11}
records the $\Pi^1_1$-consequences.

\subsection*{Acknowledgements}
The author thanks Noah Schweber for raising the question that
motivated Theorem~\ref{thm:intro-sym}, pastebee for pointing out the
resulting $\Pi^1_1$ definition of fullness, and Emil Je\v{r}\'abek for
comments on the atomic-permutation-group viewpoint. Support received
from NCN Sonata-Bis~13 (2023/50/E/ST1/00067) is gratefully
acknowledged.

\section{Framework: the descent lemma}\label{sec:descent}

We work throughout with transitive models of set theory. The generic
situation, unless stated otherwise, is a pair $M \subseteq N$ of
transitive models of $\ZF$; the main case of interest is that of a
forcing extension $N = M[G]$, but no special property of forcing is used
beyond transitivity and agreement of the membership relation. The
framework accommodates $\ZF$ ground models, symmetric submodels, inner
models and outer forcing extensions uniformly.

Given a set $Y \in M$, we write $F^M(Y)$ and $F^N(Y)$ for the standard
structure built from~$Y$ as computed in~$M$ and in~$N$ respectively. In
general one only has $F^M(Y) \subseteq F^N(Y)$, because $N$ may contain
new subsets, sequences, or operators over the old set~$Y$.

\begin{convention}[Scalars]\label{conv:scalars}
For statements involving Banach spaces, $C^*$-algebras, and other
structures whose definition makes reference to a scalar field, the
scalar field is regarded as a fixed named sort. Equivalently, we
restrict to pairs $M \subseteq N$ in which the real (and, if applicable,
complex) fields of~$M$ and~$N$ coincide. This rules out only the
orthogonal phenomenon that a ground-model real Banach space is not
automatically a vector space over the reals of a forcing extension. All
constructions below respect this convention.
\end{convention}

\begin{convention}[Analytic terminology in $\ZF$]\label{conv:analysis}
For normed and metric structures in $\ZF$ we use the choice-free
completeness convention of Blackadar--Farah--Karagila. We use the term
\emph{$\sigma$-complete} in their generic sense for the equivalent
choice-free completeness properties listed in
\cite[Theorem~1.0.2]{BlackadarFarahKaragila}. For concreteness, we
shall often use the directed-family formulation: every directed family
of non-empty closed bounded subsets, ordered by reverse inclusion and
with diameters tending to~$0$, has non-empty intersection. In the
presence of the countable axiom of choice this agrees with ordinary
Cauchy-completeness.
In this paper a Banach space means a $\sigma$-complete normed vector
space, a Hilbert space means a $\sigma$-complete inner product space,
and a $C^*$-algebra means a $\sigma$-complete normed $*$-algebra
satisfying the $C^*$-identity; cf.\ \cite[Definition~2.0.1]{BlackadarFarahKaragila}.

For a Hilbert space~$H$, a compact operator means an operator sending
bounded subsets of~$H$ to totally bounded subsets. This is the definition
advocated in \cite[Definition~7.1.2]{BlackadarFarahKaragila}; by
\cite[Proposition~7.1.4]{BlackadarFarahKaragila} it is also the
norm-closure of the finite-rank operators, a characterisation used below
for old-part absoluteness.
\end{convention}

\begin{convention}[The zero $C^*$-algebra]
We allow the zero $C^*$-algebra as a unital $C^*$-algebra, with $0 = 1$,
when this avoids irrelevant empty-space exceptions.
\end{convention}

\begin{lemma}[Descent of $\sigma$-completeness]\label{lem:sigma-complete-descends}
Let $M \subseteq N$ be transitive models of $\ZF$ satisfying the scalar
convention, and let $E \in M$ be a normed space. If $N$ sees $E$ as
$\sigma$-complete, then $M$ sees $E$ as $\sigma$-complete.
\end{lemma}

\begin{proof}
Let $(C_i)_{i \in I} \in M$ be a directed family of non-empty closed
bounded subsets of~$E$, ordered by reverse inclusion, whose diameters
tend to~$0$, all as computed in~$M$. These are absolute assertions about
the old normed space~$E$ and the old scalar field, so $N$ sees the same
family with the same properties. Since $N$ sees $E$ as $\sigma$-complete,
there is a point $x \in E$ lying in every~$C_i$. But $E \in M$ and $M$ is
transitive, hence $x \in M$. Therefore $M$ sees that the intersection of
the family is non-empty. This is exactly $\sigma$-completeness in~$M$.
\end{proof}

The elementary observation on which everything rests is the following.

\begin{lemma}[Descent lemma]\label{lem:descent}
Let $M \subseteq N$ be transitive models of $\ZF$, and let
$A, Y, \theta \in M$. Suppose that the standard construction $F$ is
absolute on old elements, in the sense that
\[
  F^M(Y) = F^N(Y) \cap M.
\]
If
\[
  N \models \text{``$\theta:A\to F^N(Y)$ is an isomorphism''},
\]
then
\[
  M \models \text{``$\theta:A\to F^M(Y)$ is an isomorphism''}.
\]
\end{lemma}

\begin{proof}
Since $\theta, A \in M$ and $M$ is transitive, every value of $\theta$
is an element of~$M$. Thus, if $N$ sees that $\theta(a) \in F^N(Y)$,
then
\[
  \theta(a) \in F^N(Y) \cap M = F^M(Y).
\]
So $M$ sees that $\theta$ maps $A$ into $F^M(Y)$.

The algebraic or metric-algebraic identities asserting that $\theta$
preserves the named operations are absolute once all objects involved
belong to~$M$. Injectivity is also absolute, since it is a bounded
statement about the ground-model set~$A$.

It remains only to check surjectivity onto the old target. Let
$z \in F^M(Y)$. Then $z \in F^N(Y)$, so in $N$ there is some $a \in A$
with $\theta(a) = z$. Since $A \in M$ and $M$ is transitive, this same
$a$ belongs to~$M$. Hence $M$ sees that $z$ lies in the range
of~$\theta$.
\end{proof}

\begin{remark}[Old-part absoluteness]\label{rem:old-part}
In all applications below, the equality $F^M(Y) = F^N(Y) \cap M$ is
straightforward. For example:
\begin{itemize}[nosep]
  \item $\Sym^M(Y) = \Sym^N(Y) \cap M$;
  \item $(Y^Y)^M = (Y^Y)^N \cap M$;
  \item $\cP^M(Y) = \cP^N(Y) \cap M$;
  \item $\Rel^M(Y) = \Rel^N(Y) \cap M$;
  \item for a fixed ground-model ring $R$,
    $(R^Y)^M = (R^Y)^N \cap M$;
  \item if $E, U \in M$, then
    $\End_E^M(U) = \End_E^N(U) \cap M$;
  \item $\mathscr{O}_Y^M = \mathscr{O}_Y^N \cap M$ for the many-sorted
    full clone on~$Y$;
  \item $\Pi^M(Y) = \Pi^N(Y) \cap M$ for the full partition lattice
    on~$Y$;
  \item under the scalar convention, old bounded, finite-support,
    $c_0$-, or $\ell_1$-families are recognised correctly in~$M$ and~$N$.
\end{itemize}
For $\cB(K)$ and $\cK(K)$, the analytic conventions above are in force.
An old function $T : K \to K$ which $N$ sees as linear and bounded is
already seen as linear and bounded in~$M$. For $\cK(K)$ on a Hilbert
space, compactness is understood as in Convention~\ref{conv:analysis};
we then use the Blackadar--Farah--Karagila characterisation
\[
  \cK(K) = \overline{\mathcal{F}(K)}^{\|\cdot\|},
\]
where $\mathcal{F}(K)$ denotes the finite-rank operators. If an old
operator $T : K \to K$ belongs to $\cK^N(K)$, then for every rational
$\varepsilon > 0$ the model $N$ contains a finite-rank operator~$S$
with
\[
  \|T - S\| < \varepsilon.
\]
By finite-dimensionality and the choice-free Riesz representation theorem
for $\sigma$-complete Hilbert spaces
\cite[Theorem~2.0.6 and Corollary~7.0.2]{BlackadarFarahKaragila}, such
an~$S$ can be written as a finite sum of rank-one operators
\[
  x \longmapsto \langle x, u_j \rangle v_j
\]
with $u_j, v_j \in K$ and old scalars; hence $S$ is coded by finite
ground-model data and belongs to~$M$. Therefore $M$ also sees~$T$ as a
norm-limit of finite-rank operators, and
\[
  \cK^M(K) = \cK^N(K) \cap M.
\]
\end{remark}

Translating to first-order-style definability, we recall that a property
$\forall Z\,\varphi(Z, A)$ with $\varphi$ absolute is downward absolute
from outer transitive models.

\begin{lemma}[$\Pi^1_1$-test]\label{lem:pi11-test}
Let $\cC$ be a class of structures definable, uniformly over transitive
models of $\ZF$, by a $\Pi^1_1$ formula of the form
$\forall Z\,\varphi(Z, A)$ whose matrix $\varphi$ is absolute for
transitive submodels. Then $\cC$ is downward absolute: if
$M \subseteq N$ are transitive models, $A \in M$, and $N \models A \in
\cC$, then $M \models A \in \cC$.
\end{lemma}

\begin{proof}
This is the standard one-line absoluteness argument. If $N \models \forall Z\,
\varphi(Z, A)$, then in particular $\varphi(Z, A)$ holds in~$N$ for
every $Z \in M$. By absoluteness of the matrix, $M \models \varphi(Z,
A)$ for every such~$Z$, so $M \models \forall Z\,\varphi(Z, A)$.
\end{proof}

Consequently, if one can exhibit $M \subseteq M[G]$ and $A \in M$ with
$A \notin \cC^M$ but $A \in \cC^{M[G]}$, then $\cC$ does not admit a
$\Pi^1_1$ definition uniform over transitive $\ZFC$-models. The
theorems of Sections~\ref{sec:zf} and~\ref{sec:zfc} show that this
strategy fails for a large collection of standardness classes, and
exhibit several natural cases where it succeeds.

\section{\texorpdfstring{Unconditional $\ZF$-descent}{Unconditional ZF-descent}}\label{sec:zf}

We gather here the standardness classes that admit canonical
reconstruction already in $\ZF$, and hence are downward absolute under
any outer transitive-model extension.

\subsection{Full symmetric groups}\label{subsec:sym}

In all that follows, $\Sym(X)$ denotes the group of all bijections
$X \to X$.

\begin{lemma}[Definability of transpositions]\label{lem:transpositions}
There is a first-order group formula $\tau(x)$ such that, for every
infinite set~$X$, the set of elements of $\Sym(X)$ satisfying $\tau$ is
exactly the set of transpositions.
\end{lemma}

\begin{proof}
Let
\[
  \tau(x) := x \ne 1 \ \wedge \ x^2 = 1 \ \wedge \ \forall y \,
  \bigl( x \cdot (y x y^{-1}) \bigr)^{6} = 1.
\]
We show that, in $\Sym(X)$ with $X$ infinite, $\tau$ defines the
transpositions. Note that $y x y^{-1}$ ranges over the full conjugacy
class of~$x$ as~$y$ ranges over $\Sym(X)$, so the clause says that $x$
multiplied by any of its conjugates has order dividing~$6$.

\emph{If $x$ is a transposition}, then every conjugate $z = yxy^{-1}$
is also a transposition, and the product of two transpositions has
order $1$ (when they are equal), $2$ (when they are disjoint), or $3$
(when they share a point). Hence $(xz)^6 = 1$.

\emph{Conversely}, suppose $x \in \Sym(X)$ is an involution moving at
least four points and not a transposition. We shall produce a conjugate
$z = yxy^{-1}$ with $(xz)^6 \ne 1$.

Write $x = \prod_{i \in I} (a_i\, b_i)$ as a product of disjoint
$2$-cycles, with $|I| \geqslant 2$, the $a_i, b_i \in X$ pairwise
distinct across~$i$.

\emph{Case A.} $x$ has a fixed point $e \in X$. Since $|I| \geqslant 2$,
pick indices $1, 2 \in I$, and consider the $5$ distinct points $a_1,
b_1, a_2, b_2, e$. Define $\sigma \in \Sym(X)$ to act on these five
points as the $5$-cycle $(a_1\,b_1\,a_2\,b_2\,e)$ and as the identity
on every $2$-cycle $\{a_i, b_i\}$ of~$x$ with $i \geqslant 3$ and
everywhere else. Set $z = \sigma x \sigma^{-1}$. A direct computation
shows that $z$ acts on $\{a_1, b_1, a_2, b_2, e\}$ as
$(b_1\,a_2)(b_2\,e)$ and agrees with $x$ elsewhere. Composing,
\[
  xz \text{ acts on } \{a_1, b_1, a_2, b_2, e\} \text{ as }
  (a_1\,b_1\,b_2\,e\,a_2),
\]
a $5$-cycle, and as the identity off this set. Hence $xz$ has order~$5$
and $(xz)^6 = xz \ne 1$.

\emph{Case B.} $x$ has no fixed points. Since $X$ is infinite and $x$
has no fixed point, $x$ has infinitely many $2$-cycles; in particular,
it has at least four of them. Pick indices $1, 2, 3, 4 \in I$ and
consider the $8$ distinct points $a_1, b_1, \ldots, a_4, b_4$. Define
$\sigma \in \Sym(X)$ to act on these eight points as the $8$-cycle
$(a_1\, b_1\, a_2\, b_2\, a_3\, b_3\, a_4\, b_4)$, and as the identity
on every other $2$-cycle of~$x$. Set $z = \sigma x \sigma^{-1}$. Then
$z$ restricted to $\{a_1, b_1, \ldots, a_4, b_4\}$ is the involution
$(b_1\, a_2)(b_2\, a_3)(b_3\, a_4)(b_4\, a_1)$, and agrees with~$x$
elsewhere. A direct computation shows
\[
  xz \text{ acts on } \{a_1, b_1, \ldots, a_4, b_4\} \text{ as }
  (a_1\, a_4\, a_3\, a_2)(b_1\, b_2\, b_3\, b_4),
\]
\emph{i.e.}\ as a product of two disjoint $4$-cycles, and as the identity
elsewhere. Hence $xz$ has order~$4$, and $(xz)^6 = (xz)^2 =
(a_1\,a_3)(a_2\,a_4)(b_1\,b_3)(b_2\,b_4) \ne 1$.

In both cases $\tau(x)$ fails, completing the proof.
\end{proof}

\begin{remark}
For finite $X$, the same formula need not isolate the transpositions:
in $S_4$ it also detects double transpositions, and in $S_6$ it detects
the fixed-point-free involutions. This is irrelevant for
Theorem~\ref{thm:symmetric-groups}, since the finite case is handled
separately. For $|X| = 0, 1, 2, 3, 5$ the formula does isolate the
transpositions. The broader first-order analysis of infinite symmetric
groups, including first-order control of their natural permutation-group
structure, goes back in various formulations to
Shelah~\cite{Shelah-permgroups} and McKenzie~\cite{McKenzie}; see also
\cite{Dixon-Mortimer,Scott} and the references therein.
\end{remark}

\begin{theorem}[Canonical reconstruction for $\Sym(X)$]
\label{thm:symmetric-groups}
Let $M \subseteq N$ be transitive models of $\ZF$ and let $G \in M$ be
a group. If
\[
  N \models \text{``$G \cong \Sym(X)$ for some set $X$''},
\]
then
\[
  M \models \text{``$G \cong \Sym(Y)$ for some set $Y$''}.
\]
\end{theorem}

\begin{proof}
If $N$ sees $G \cong \Sym(X)$ with $X$ finite, then $G$ is finite and
the finite isomorphism is already an element of~$M$, which therefore
witnesses the conclusion. We henceforth assume $X$ infinite in~$N$.

\emph{Reconstruction of transpositions.}
By Lemma~\ref{lem:transpositions} there is a formula $\tau(x)$
defining transpositions in $\Sym(X)$. In~$M$, let
\[
  T = \{g \in G : G \models \tau(g)\}.
\]
Since the definition of~$T$ is a first-order property of~$G$, the set
$T$ is an element of~$M$. In~$N$, under any isomorphism $G \cong
\Sym(X)$, $T$ is identified with the full set of transpositions of~$X$.

\emph{Reconstruction of the underlying set.}
For non-commuting $s, t \in T$, define
\[
  P(s, t) = \{s, t\} \cup \bigl\{ r \in T :
  rs \ne sr, \ rt \ne tr, \ r \ne sts \bigr\}.
\]
If, in $\Sym(X)$, $s = (a\, b)$ and $t = (a\, c)$ with $a, b, c$
pairwise distinct, then $sts = (b\, c)$, and a transposition~$r$ fails
to commute with both~$s$ and~$t$ precisely when $r = (a\, d)$ for some
$d \in X \setminus \{a\}$, or $r = (b\, c)$. Thus
\[
  P(s, t) = \bigl\{ (a\, d) : d \in X, \ d \ne a \bigr\} =: E_a,
\]
the \emph{star of~$a$}.

Let
\[
  Y = \bigl\{ P(s, t) : s, t \in T, \ st \ne ts \bigr\} \in M.
\]
Two stars $E_a$ and $E_{a'}$ with $a \ne a'$ are distinct, and the map
$a \mapsto E_a$ is a bijection $X \to Y$ in~$N$. Thus $Y$ is a canonical
avatar of the original set~$X$, reconstructed from~$G$ alone.

\emph{Reconstruction of the action.}
Define, in~$M$,
\[
  \theta : G \to \Sym^M(Y), \qquad
  \theta(g)(E) = \{ g r g^{-1} : r \in E \} \quad (E \in Y).
\]
Conjugation by~$g$ preserves transpositions and non-commutation, and
hence permutes the stars. For each $g \in G$, $\theta(g)$ is thus a
permutation of~$Y$ belonging to~$M$. In~$N$, after identification with
$\Sym(X)$, the map $\theta$ is the natural action of $\Sym(X)$ on its
set of stars, which is isomorphic to the tautological action on~$X$;
in particular, $\theta$ is a group isomorphism $G \to \Sym^N(Y)$.

By Lemma~\ref{lem:descent} applied to $A = G$, using $\Sym^M(Y) =
\Sym^N(Y) \cap M$, the same $\theta$ is an isomorphism $G \to
\Sym^M(Y)$ in~$M$. Hence $M$ sees $G \cong \Sym(Y)$.
\end{proof}

\begin{corollary}[A uniform $\Pi^1_1$ definition of fullness]
\label{cor:sym-pi11}
The class of groups isomorphic to $\Sym(X)$ for some set~$X$ admits a
uniform $\Pi^1_1$ definition over transitive models of $\ZF$.
\end{corollary}

\begin{proof}
For each $n < 8$, the assertion $G \cong S_n$ is first-order, since it
is the assertion that $G$ is a finite group with a specified
multiplication table. It remains to describe a single $\Pi^1_1$
condition capturing the cases $|X| \geqslant 8$ and the infinite cases.

Let $\tau(x)$ be the group formula from
Lemma~\ref{lem:transpositions}. The same proof shows that $\tau$
defines the transpositions in $\Sym(X)$ whenever $|X| \geqslant 8$, as
well as when $X$ is infinite. For $s, t \in G$, write $D(s, t)$ for the
first-order condition
\[
  \tau(s) \wedge \tau(t) \wedge st \ne ts.
\]
For pairs $(s, t)$ satisfying $D$, define the first-order relation
$\rho(r; s, t)$ by
\[
  \rho(r; s, t)
  \Longleftrightarrow
  \tau(r) \wedge
  \bigl( r = s \vee r = t \vee
    (rs \ne sr \wedge rt \ne tr \wedge r \ne sts) \bigr).
\]
Thus, in an actual symmetric group, $\rho(r; s, t)$ says that the
transposition~$r$ belongs to the star $P(s, t)$.

Define an equivalence relation on pairs satisfying~$D$ by
\[
  (s, t) \equiv (u, v)
  \Longleftrightarrow
  D(s, t) \wedge D(u, v) \wedge
  \forall r\, \bigl( \rho(r; s, t) \leftrightarrow \rho(r; u, v) \bigr).
\]
Let $Y$ be the quotient of the definable class of $D$-pairs by this
equivalence relation. This is not introduced as a new parameter in the
formula; it is only a convenient abbreviation for the definable
quotient.

Conjugation gives a definable action on the quotient via
\[
  g \cdot (s, t) = (g s g^{-1}, g t g^{-1}).
\]
Now let $R$ be a second-order variable coding a binary relation on the
set of $D$-pairs. There is a first-order condition, with parameter~$R$,
saying that $R$ codes a permutation of the quotient $Y = D/{\equiv}$:
namely, $R$ is $\equiv$-saturated, total and single-valued
on~$\equiv$-classes, and likewise for the inverse relation. Denote this
first-order condition by $\operatorname{Perm}(R)$. Likewise, for
$g \in G$, there is a first-order condition $\operatorname{Ind}(g, R)$
saying that $R$ is the permutation induced by~$g$:
\[
  \forall (s, t)\, \forall (u, v)\,
  \Bigl(
  R\bigl((s, t), (u, v)\bigr)
  \leftrightarrow
  (u, v) \equiv (g s g^{-1}, g t g^{-1})
  \Bigr),
\]
with the quantifiers restricted to $D$-pairs.

Consider the sentence $\Phi_\infty(G)$ asserting
\[
  \forall g\,
  \Bigl[
    \bigl(
      \forall (s, t)\,
      \bigl(
        D(s, t) \Rightarrow
        (g s g^{-1}, g t g^{-1}) \equiv (s, t)
      \bigr)
    \bigr)
    \Rightarrow g = 1
  \Bigr]
\]
and
\[
  \forall R\,
  \bigl(
    \operatorname{Perm}(R) \Rightarrow
    \exists g \in G\ \operatorname{Ind}(g, R)
  \bigr).
\]
The first displayed part says that the conjugation action on the
reconstructed quotient~$Y$ is faithful; the second says that every
permutation of~$Y$ is induced by an element of~$G$. This is a
$\Pi^1_1$ condition: the only second-order quantifier is the universal
quantifier over~$R$; all remaining quantifiers are first-order
quantifiers over elements of~$G$.

Now suppose $G \models \Phi_\infty$. In the ambient transitive model,
form the actual quotient set $Y = D/{\equiv}$. The conjugation action
gives a homomorphism $\theta : G \to \Sym(Y)$; the first part of
$\Phi_\infty$ says that~$\theta$ is injective, and the second part says
that~$\theta$ is surjective. Hence $G \cong \Sym(Y)$.

Conversely, if $G \cong \Sym(X)$ with $|X| \geqslant 8$ or with~$X$
infinite, then $\tau$ defines precisely the transpositions, the
quotient $D/{\equiv}$ is exactly the set of point-stars, and the
conjugation action is the natural full action of $\Sym(X)$ on~$X$.
Thus $G \models \Phi_\infty$.

Therefore fullness is defined by the single formula
\[
  \Bigl(\bigvee_{n < 8} G \cong S_n\Bigr) \vee \Phi_\infty(G).
\]
Since a finite first-order disjunction with a $\Pi^1_1$ formula is
again equivalent to a $\Pi^1_1$ formula, this gives the desired uniform
$\Pi^1_1$ definition.
\end{proof}

\begin{remark}[Atomic permutation groups]
The preceding proof separates the first-order and genuinely
second-order parts of fullness. The first-order part reconstructs a
definable quotient~$Y$ and says that~$G$ acts faithfully on~$Y$ by
conjugation. If one adds the first-order clauses saying that the
$\tau$-elements act as transpositions on~$Y$ and that every pair of
distinct points of~$Y$ is swapped by some $\tau$-element, one obtains
the usual ``atomic permutation group'' situation: $G$ is identified
with a subgroup of $\Sym(Y)$ containing all transpositions. Fullness is
then exactly the additional $\Pi^1_1$ assertion that every permutation
of~$Y$ is induced by an element of~$G$.
\end{remark}

\subsection{Full transformation monoids}\label{subsec:trans-monoid}

Let $X^X$ denote the full transformation monoid under composition.

\begin{theorem}\label{thm:transformation-monoid}
Let $M \subseteq N$ be transitive models of $\ZF$ and let $S \in M$ be
a monoid. If $N \models$ ``$S \cong X^X$ for some set $X$'', then
$M \models$ ``$S \cong Y^Y$ for some set $Y$''.
\end{theorem}

\begin{proof}
Assume $X \ne \emptyset$ (the empty case is trivial). In $X^X$, an
element $c$ is a constant function if and only if $c \circ f = c$ for
every $f \in X^X$: necessity is immediate; sufficiency follows by
taking~$f$ itself constant. Hence the set
\[
  C = \{ s \in S : \forall f \in S,\ sf = s \} \in M
\]
corresponds in~$N$ to the set of constants in~$X^X$, which is a copy
of~$X$.

Define, in~$M$,
\[
  \theta : S \to C^C, \qquad \theta(s)(c) = sc \quad (c \in C).
\]
Since constants are closed under left multiplication by any element,
$\theta$ is well-defined, and in~$N$ it coincides with the standard
identification of~$X^X$ with the full transformation monoid on its set
of constants. By Lemma~\ref{lem:descent}, $\theta$ is an isomorphism
$S \to C^C$ in~$M$.
\end{proof}

\subsection{Powerset Boolean algebras}\label{subsec:powersets}

\begin{theorem}\label{thm:powersets}
Let $M \subseteq N$ be transitive models of $\ZF$ and let $B \in M$ be
a Boolean algebra (equivalently, a Boolean ring). If $N \models$
``$B \cong \cP(X)$ for some set $X$'', then $M \models$ ``$B \cong
\cP(Y)$ for some set $Y$''.
\end{theorem}

\begin{proof}
Let $Y = \At(B) \in M$ be the set of atoms of~$B$, defined by the
absolute formula ``$a \ne 0$ and, for every $b$, $0 \leqslant b
\leqslant a$ implies $b = 0$ or $b = a$''. For $b \in B$ put
$\theta(b) = \{ a \in Y : a \leqslant b \}$. In~$N$ this is the
canonical atomic representation of the powerset algebra~$\cP(X)$,
identifying $\At(B)$ with~$X$. By Lemma~\ref{lem:descent}, $\theta$ is
an isomorphism $B \to \cP^M(Y)$ in~$M$.
\end{proof}

\begin{corollary}\label{cor:complete-atomic}
The property ``$B$ is a complete atomic Boolean algebra'' is downward
absolute between transitive models of $\ZF$.
\end{corollary}

\begin{proof}
In $\ZF$, every complete atomic Boolean algebra is canonically
isomorphic to the powerset algebra of its set of atoms. Indeed, if
$Y = \At(B)$, the map
\[
  b \longmapsto \{ a \in Y : a \leqslant b \}
\]
is injective by atomicity and surjective by completeness, since every
set of atoms has a supremum.

Thus, if $N$ sees that $B$ is complete atomic, then $N$ sees
$B \cong \cP(Y)$ for its set of atoms. By
Theorem~\ref{thm:powersets}, $M$ sees $B \cong \cP(Z)$ for some
set~$Z$, and hence $M$ sees that $B$ is complete atomic.
\end{proof}

\subsection{Full relation algebras}\label{subsec:relation-algebras}

Let $\Rel(X) = \cP(X \times X)$ be the full relation algebra on $X$,
with Boolean operations, relational composition
$(R; S) = \{(x, z) : \exists y\, (x, y) \in R \wedge (y, z) \in S\}$,
converse $R^{\smile} = \{(y, x) : (x, y) \in R\}$, and identity relation
$1' = \{(x, x) : x \in X\}$.

\begin{theorem}\label{thm:relation-algebras}
Let $M \subseteq N$ be transitive models of $\ZF$, and let $A \in M$ be
a relation algebra. If
\[
  N \models \text{``$A \cong \Rel(X)$ for some set $X$''},
\]
then
\[
  M \models \text{``$A \cong \Rel(Y)$ for some set $Y$''}.
\]
\end{theorem}

\begin{proof}
In $\Rel(X)$, the atoms of the underlying Boolean algebra are the
singletons $\{(x, y)\}$. Among them, the atoms below the identity
relation $1'$ are precisely the diagonal atoms $\{(x, x)\}$. Thus the
set
\[
  Y = \{ a \in A : a \text{ is a Boolean atom and }
  a \leqslant 1' \}
\]
is definable in the relation-algebra structure and is a canonical copy
of~$X$.

For $r \in A$, define
\[
  \theta(r) = \{ (p, q) \in Y \times Y : p ; r ; q \ne 0 \}.
\]
In~$N$, after identifying $A \cong \Rel(X)$ and $Y$ with $X$, this is
exactly the usual identification of a relation with the set of pairs of
points it relates. Hence $N$ sees
\[
  \theta : A \to \Rel^N(Y)
\]
as an isomorphism. Since $\Rel^M(Y) = \Rel^N(Y) \cap M$, the descent
lemma gives the result in~$M$.
\end{proof}

\subsection{Full clones}\label{subsec:full-clones}

Let $\mathscr{O}_X$ denote the many-sorted clone of all finitary
operations on a set~$X$, with sorts $X^{X^n}$ for $n < \omega$,
composition, and projections; cf.~\cite{Szendrei}. For $X = \emptyset$
we use the usual convention: the nullary sort is empty, and every
positive-arity sort is a singleton.

\begin{theorem}\label{thm:full-clones}
Let $M \subseteq N$ be transitive models of $\ZF$, and let $C \in M$ be
a many-sorted clone. If
\[
  N \models \text{``$C \cong \mathscr{O}_X$ for some set $X$''},
\]
then
\[
  M \models \text{``$C \cong \mathscr{O}_Y$ for some set $Y$''}.
\]
\end{theorem}

\begin{proof}
If $N$ sees $X = \emptyset$, then $C$ is the degenerate many-sorted
clone with empty nullary sort and exactly one operation in each
positive arity. This is absolute sort by sort: each sort is either
empty or a singleton, and the unique sort-preserving isomorphism
belongs to~$M$. Thus $M$ already sees $C \cong \mathscr{O}_{\emptyset}$.
Assume henceforth that $X \ne \emptyset$.

The constant unary operations form a definable copy of the underlying
set: in $\mathscr{O}_X$, a unary operation $c$ is constant iff $c$
equals the composition $c \circ f$ for every unary operation~$f$. Let
$Y$ be this set. Every $n$-ary operation acts on $Y^n$ by substitution
into constants. This gives, in~$M$, a clone homomorphism
\[
  \theta : C \to \mathscr{O}_Y.
\]
In~$N$, this is the standard representation of the full clone on its
set of constants, hence an isomorphism onto $\mathscr{O}_Y^N$. Since
old finitary operations on~$Y$ are exactly the old elements of
$\mathscr{O}_Y^N$, Lemma~\ref{lem:descent} applies.
\end{proof}

\subsection{Full partition lattices}\label{subsec:partition-lattices}

Let $\Pi(X)$ denote the lattice of all equivalence relations on~$X$,
ordered by refinement, with join equal to the equivalence hull of the
union.

\begin{theorem}\label{thm:partition-lattices}
Let $M \subseteq N$ be transitive models of $\ZF$, and let $L \in M$ be
a lattice. If
\[
  N \models \text{``$L \cong \Pi(X)$ for some set $X$''},
\]
then
\[
  M \models \text{``$L \cong \Pi(Y)$ for some set $Y$''}.
\]
\end{theorem}

\begin{proof}
The cases $|X| \leqslant 2$ are finite and hence absolute. Assume
$|X| \geqslant 3$. In $\Pi(X)$, the atoms are precisely the partitions
obtained from the discrete partition by identifying a single unordered
pair $\{a, b\}$. Thus the atom set of~$L$ is, in~$N$, canonically
identified with the edge set of the complete graph on~$X$. Throughout
the proof we write $\{a, b\}$ for the atom whose non-trivial block is
$\{a, b\}$.

Say that two atoms $s, t$ are \emph{triangular} if $s \ne t$ and their
join $s \vee t$ dominates exactly three atoms. In $\Pi(X)$, this is
equivalent to $s$ and $t$ sharing a point: if $s = \{a, b\}$ and $t =
\{a, c\}$, then $s \vee t$ is the partition whose non-trivial block is
$\{a, b, c\}$ and the three atoms below $s \vee t$ are $\{a, b\}$,
$\{a, c\}$, $\{b, c\}$; while if $s = \{a, b\}$ and $t = \{c, d\}$ with
$\{a, b\} \cap \{c, d\} = \emptyset$, then the atoms below $s \vee t$
are only $s$ and~$t$.

For two triangular atoms $s, t$, let $w(s, t)$ denote the unique atom
below $s \vee t$ different from both~$s$ and~$t$ (in the example above,
$w(s, t) = \{b, c\}$). Define, in complete analogy with the
symmetric-group construction,
\[
  P(s, t) = \{s, t\} \cup \bigl\{ r : r \text{ triangular with both }
  s \text{ and } t,\ r \ne w(s, t) \bigr\}.
\]
If $s = \{a, b\}$ and $t = \{a, c\}$, then an atom~$r$ is triangular
with both~$s$ and~$t$ iff $r$ shares a point with each, \emph{i.e.}\
$r \in \{\{a, x\} : x \ne a\} \cup \{\{b, c\}\}$. Excluding $w(s, t) =
\{b, c\}$, we obtain the star
\[
  P(s, t) = \bigl\{ \{a, x\} : x \ne a \bigr\} =: E_a.
\]

Set
\[
  Y = \bigl\{ P(s, t) : s, t \text{ triangular atoms of } L \bigr\}
  \in M.
\]
In~$N$, $Y$ is canonically bijective with~$X$ via $a \mapsto E_a$.

Finally, recover the lattice structure. If $E, E' \in Y$ are distinct
stars, then $E \cap E'$ consists of a unique atom; in~$N$, if $E = E_a$
and $E' = E_b$, this atom is $\{a, b\}$.

For $\ell \in L$, define a binary relation $R_\ell$ on $Y$ by
\[
  E \, R_\ell \, E'
\]
iff either $E = E'$, or $E \ne E'$ and the unique atom in $E \cap E'$
lies below~$\ell$. Let $\theta(\ell)$ be the transitive closure
of~$R_\ell$. Since $R_\ell$ is symmetric and reflexive by construction,
its transitive closure is an equivalence relation on~$Y$.

In~$N$, this is exactly the usual description of a partition by the
graph whose edges are the pairs contained in a common block. Hence $N$
sees
\[
  \theta : L \to \Pi^N(Y)
\]
as an isomorphism. Since $\Pi^M(Y) = \Pi^N(Y) \cap M$,
Lemma~\ref{lem:descent} yields the conclusion in~$M$.
\end{proof}

\subsection{Products of finitely generated centrally indecomposable rings}
\label{subsec:fg-products}

The following $\ZF$-theorem covers many natural finitely generated
ring factors, including finite factors with non-trivial automorphisms.

\begin{theorem}\label{thm:fg-products}
Let $M \subseteq N$ be transitive models of $\ZF$. Let $R \in M$ be a
non-zero unital ring with no central idempotents other than $0$ and~$1$.
Assume that $R$ is generated as a unital ring by a finite tuple
\[
  \bar r = (r_1, \ldots, r_m).
\]
Let $A \in M$ be a unital ring. If
\[
  N \models \text{``$A \cong R^X$ for some set $X$''},
\]
then
\[
  M \models \text{``$A \cong R^Y$ for some set $Y$''}.
\]
\end{theorem}

\begin{proof}
The case $X = \emptyset$ is trivial, so assume $X \ne \emptyset$
in~$N$. Let $E$ be the Boolean algebra of central idempotents of~$A$,
and let
\[
  Y = \At(E) \in M.
\]
Since $R$ has no non-trivial central idempotents, $N$ identifies $Y$
with the set of coordinate idempotents of~$R^X$.

Choose, in~$N$, an isomorphism $j : A \to R^X$. For each generator
$r_i$, let $a_i \in A$ be the unique element such that $j(a_i)$ is the
constant function on~$X$ with value~$r_i$. Since $A \in M$, each $a_i$
belongs to~$M$.

For every $e \in Y$ we now define a homomorphism
\[
  \phi_e : R \longrightarrow Ae.
\]
Let $\cT_m$ be the set of ring terms in $m$ variables and integer
coefficients, coded by natural numbers in the usual way. Since $\bar r$
generates~$R$ as a unital ring, for every $r \in R$ there is a term
$t \in \cT_m$ such that
\[
  r = t^R(\bar r).
\]
Choose the least such term in the fixed coding, and define
\[
  \phi_e(r) = t^A(\bar a)\, e,
\]
where $\bar a = (a_1, \ldots, a_m)$.

This is independent of the choice of representative term. Indeed, if
\[
  t^R(\bar r) = u^R(\bar r),
\]
then, in~$N$, after applying the isomorphism $j : A \to R^X$, the
elements $t^A(\bar a)\, e$ and $u^A(\bar a)\, e$ have the same value in
the coordinate corresponding to~$e$. Hence they are equal in~$Ae$.
Since all objects involved lie in~$M$, this equality is absolute.
Therefore $\phi_e$ is a well-defined unital ring homomorphism in~$M$.

In~$N$, the map $\phi_e$ is precisely the coordinate identification
of~$R$ with the corner~$Ae$, and therefore is an isomorphism. Since
$R, A, e$ and the finite tuple $\bar a$ all belong to~$M$, the
assertion that $\phi_e$ is a ring isomorphism is absolute; thus $M$
also sees that every $\phi_e$ is an isomorphism.

Define, in~$M$,
\[
  \theta : A \longrightarrow R^Y, \qquad
  \theta(a)(e) = \phi_e^{-1}(a e) \qquad (a \in A,\ e \in Y).
\]
In~$N$, this is the usual coordinate map $A \cong R^Y$. Hence $N$ sees
$\theta$ as an isomorphism $A \to (R^Y)^N$. By
Lemma~\ref{lem:descent}, using $(R^Y)^M = (R^Y)^N \cap M$, the same map
is an isomorphism $A \to (R^Y)^M$ in~$M$.
\end{proof}

\begin{corollary}\label{cor:fg-product-examples}
In $\ZF$, the property ``$A \cong R^X$ for some set $X$'' is downward
absolute for each of the following fixed rings~$R$:
\[
  \bbZ, \qquad
  \bbZ/p^m\bbZ, \qquad
  \bbF_{p^n}, \qquad
  M_k(\bbF_q), \qquad
  \bbZ[t_1, \ldots, t_n].
\]
More generally, the same holds for every finitely generated unital ring
with no non-trivial central idempotents.
\end{corollary}

\begin{remark}
Products $R^X$ with $R$ a finite non-rigid factor — a finite field
$\bbF_{p^n}$ with $n \geqslant 2$, or a matrix ring $M_k(\bbF_q)$ — are
\emph{not} merely $\ZFC$-descent examples: they descend already in~$\ZF$.
The key point is that the preimages of the finitely many constant
generators of~$R$ globally trivialise all coordinate corners at once.
Thus the usual automorphism-torsor obstruction does not apply to full
products $R^X$ of finitely generated factors. (The torsor phenomenon
does survive for the more delicate finite-cover and finite-support
examples of Section~\ref{sec:zfc}.)
\end{remark}

\subsection{\texorpdfstring{Atomic commutative $C^*$-algebras}{Atomic commutative C*-algebras}}\label{subsec:cstar}

\begin{theorem}\label{thm:cstar-commutative}
Under the scalar convention, the following properties are downward
absolute in $\ZF$:
\begin{enumerate}[label=\textup{(\roman*)},nosep]
  \item being isomorphic, as a unital commutative $C^*$-algebra, to
    $\ell_\infty(X)$ for some set~$X$;
  \item being isomorphic, as a commutative $C^*$-algebra (not
    necessarily unital), to $c_0(X)$ for some set~$X$.
\end{enumerate}
\end{theorem}

\begin{proof}
We first handle the unital case. Let $A \in M$ be a unital commutative
$C^*$-algebra such that $N$ sees $A \cong \ell_\infty(X)$. Let $Y
\subseteq A$ be the set of minimal non-zero projections, definable in
the $C^*$-language from the order on self-adjoint elements. For $a \in
A$ and $p \in Y$, the element $ap$ lies in the one-dimensional corner
$pAp = \bbC p$, so there is a unique scalar $\lambda_p(a)$ with
\[
  a p = \lambda_p(a)\, p.
\]
Define
\[
  \theta(a)(p) = \lambda_p(a) \qquad (a \in A,\ p \in Y).
\]
In~$N$, this is the canonical isomorphism $\ell_\infty(X) \to
\ell_\infty(Y)$. Hence Lemma~\ref{lem:descent} gives an isomorphism
$A \cong \ell_\infty^M(Y)$ in~$M$.

For the $c_0$ case the same set $Y$ of minimal projections is used.
The same coordinate map is, in~$N$, the canonical isomorphism
\[
  A \to c_0^N(Y).
\]
By old-part absoluteness for $c_0$-families, Lemma~\ref{lem:descent}
gives $A \cong c_0^M(Y)$ in~$M$.
\end{proof}

\begin{remark}[General commutative $C^*$-algebras]
The atomic cases $\ell_\infty(X)$ and $c_0(X)$ are safe because the
points are recovered as minimal projections, and because the standard
algebras involved are formed using the named scalar field and
$\sigma$-complete normed $*$-algebras as in
Convention~\ref{conv:analysis}. For a general compact Hausdorff
space~$K$, the assertion that a ground-model algebra becomes isomorphic
to $C(K)$ in an outer model involves possible new continuous functions
and a more delicate comparison of spectra; we do not use such a general
statement here.
\end{remark}

\subsection{Endomorphism rings}\label{subsec:endo}

We now pass to reconstructions where the recovered object is not a set
but a module or Hilbert space on which the algebra acts.

\begin{theorem}\label{thm:endomorphism-rings}
Let $M \subseteq N$ be transitive models of $\ZF$, and let $A \in M$ be
a unital ring. Suppose
\[
  N \models \text{``$A \cong \End_D(V)$''}
\]
for some non-zero right vector space $V$ over some division ring~$D$,
and that~$V$ possesses a one-dimensional complemented subspace in~$N$.
Then there exist $E, U \in M$ such that
\[
  M \models \text{``$E$ is a division ring and
  $A \cong \End_E(U)$ as rings''}.
\]
\end{theorem}

\begin{proof}
In~$N$, the hypothesis furnishes a rank-one idempotent $p \in A$. Since
$p \in A \in M$, $p \in M$.

Set $E = pAp$ and $U = Ap$ in~$M$. In~$N$, the idempotent~$p$ is a
rank-one idempotent in a full endomorphism ring. Hence $pAp$ is a
division ring, $Ap$ is the associated column module, and the standard
Morita identification gives
\[
  A \cong \End_{pAp}(Ap)
\]
by left multiplication. Absoluteness thus yields, in~$M$, that $E$ is a
division ring and $U$ is a right $E$-module.

Define
\[
  \lambda : A \to \End_E(U), \qquad \lambda(a)(u) = au
\]
in~$M$. In~$N$, $\lambda$ realises the Morita identification above, and
is an isomorphism. Using old-part absoluteness for $\End_E(U)$,
Lemma~\ref{lem:descent} transports this conclusion to~$M$.
\end{proof}

\begin{remark}[Working without a named division ring]
Theorem~\ref{thm:endomorphism-rings} reconstructs \emph{a} division
ring~$E$ together with a representation, but does not assert that $E$
agrees with a specified ground-model division ring. If the division
ring is \emph{named}, \emph{e.g.}, $D = \bbQ$ or $D = \bbF_q$, one needs
moreover that $E = pAp$ is $\ZF$-absolutely isomorphic to the fixed~$D$;
this is automatic if~$D$ is finitely generated and centrally
indecomposable in the sense of Theorem~\ref{thm:fg-products}, and in
particular for finite fields. It is also automatic for the prime
field~$\bbQ$, where the isomorphism is the canonical map
$q \mapsto q \cdot 1_E$.
\end{remark}

\begin{corollary}\label{cor:matrix-rings}
In $\ZF$, the following are downward absolute:
\begin{enumerate}[label=\textup{(\roman*)},nosep]
  \item isomorphism with some full matrix ring $M_n(D)$, where
    $n \geqslant 1$ and the division ring~$D$ are allowed to vary;
  \item isomorphism with some $\End_D(V)$ with~$V$ admitting a
    complemented line.
\end{enumerate}
\end{corollary}

\begin{proof}
Part~\textup{(ii)} is Theorem~\ref{thm:endomorphism-rings}.

For~\textup{(i)}, suppose that $N$ sees $A \cong M_n(D)$ for some
$n \geqslant 1$ and some division ring~$D$. Applying
Theorem~\ref{thm:endomorphism-rings}, $M$ obtains a division ring
$E = pAp$ and a right $E$-module $U = Ap$ such that
\[
  A \cong \End_E(U).
\]
In~$N$, the module~$U$ has an $E$-basis of size~$n$. Choose such a
finite basis in~$N$. Its elements belong to the ground-model set~$U$,
so the finite tuple itself belongs to~$M$, and the assertion that it is
an $E$-basis is absolute. Hence $M$ sees $U \cong E^n$, and therefore
\[
  A \cong \End_E(E^n) \cong M_n(E). \qedhere
\]
\end{proof}

\subsection{\texorpdfstring{The operator algebras $\cB(H)$ and $\cK(H)$}{The operator algebras B(H) and K(H)}}\label{subsec:BH}

\begin{theorem}\label{thm:BH}
Let $M \subseteq N$ be transitive models of $\ZF$, and let $A \in M$ be
a unital complex $C^*$-algebra. If $N \models$ ``$A \cong \cB(H)$ for
some Hilbert space~$H$'', then $M \models$ ``$A \cong \cB(K)$ for some
Hilbert space~$K$''.
\end{theorem}

\begin{proof}
If $N$ sees $H = 0$, then $A$ is the zero $C^*$-algebra, and the
conclusion is immediate. Assume henceforth that $H \ne 0$. In~$N$,
$A \cong \cB(H)$ contains a rank-one (\emph{i.e.}\ minimal) projection~$p$,
which lies already in~$M$.

Set $K = Ap$ in~$M$. For $x, y \in K$ we have $x^* y \in pAp$.
Minimality of~$p$ yields $pAp = \bbC p$ (this is absolute once~$p$ is
minimal). Hence there is a unique scalar $\langle x, y \rangle$ with
\[
  x^* y = \langle x, y \rangle p.
\]
The formula above defines an inner product on $K = Ap$ in~$M$. In~$N$,
after identifying $A$ with $\cB(H)$ and $p$ with a rank-one projection,
the Hilbert space $K = Ap$ is naturally isometric to~$H$. Thus $N$ sees
$K$ as a $\sigma$-complete Hilbert space. By
Lemma~\ref{lem:sigma-complete-descends}, the same $K$ is a Hilbert
space in~$M$.

Finally, left multiplication
\[
  \theta : A \to \cB^M(K), \qquad \theta(a)(x) = ax
\]
is a $*$-homomorphism in~$M$. In~$N$, $\theta$ is the canonical
faithful representation of $\cB(H)$ on $H \cong Ap$, which is an
isomorphism onto $\cB(K)$. Apply Lemma~\ref{lem:descent}.
\end{proof}

\begin{theorem}\label{thm:KH}
Let $M \subseteq N$ be transitive models of $\ZF$, and let $A \in M$ be
a complex $C^*$-algebra. If $N \models$ ``$A \cong \cK(H)$ for some
Hilbert space~$H$'', then $M \models$ ``$A \cong \cK(K)$ for some
Hilbert space~$K$''.
\end{theorem}

\begin{proof}
Assume $H \ne 0$; the zero case is trivial. In~$N$, choose a minimal
projection $p \in A$. Since $A \in M$, this projection belongs to~$M$.
As in the proof of Theorem~\ref{thm:BH}, define $K = Ap$ and give it
the inner product determined by
\[
  x^* y = \langle x, y \rangle p.
\]
Then $N$ identifies $K$ with~$H$, and in particular sees $K$ as a
$\sigma$-complete Hilbert space. By
Lemma~\ref{lem:sigma-complete-descends}, $M$ sees the same $K$ as a
Hilbert space.

Left multiplication defines
\[
  \theta : A \longrightarrow \cB(K), \qquad \theta(a)(x) = ax.
\]
In~$N$, under the identifications $A \cong \cK(H)$ and $K \cong H$, the
range of~$\theta$ is exactly $\cK(K)^N$. Therefore $N$ sees
\[
  \theta : A \to \cK(K)^N
\]
as a $C^*$-isomorphism.

By the old-part absoluteness for $\cK(K)$ explained in
Remark~\ref{rem:old-part}, we have
\[
  \cK^M(K) = \cK^N(K) \cap M.
\]
Lemma~\ref{lem:descent} therefore yields
\[
  M \models A \cong \cK(K)^M. \qedhere
\]
\end{proof}

\subsection{\texorpdfstring{$\ell_1$ as a Banach lattice}{ell_1 as a Banach lattice}}\label{subsec:l1-lattice}

Recall that a Banach lattice $\ell_1(\Gamma)$ with its coordinatewise
order has a distinguished set of \emph{positive normalised atoms}:
unit vectors $u \geqslant 0$ with $\|u\| = 1$ such that $0 \leqslant v
\leqslant u$ implies $v = \alpha u$ for some $0 \leqslant \alpha
\leqslant 1$. These are precisely the standard basis vectors~$e_\gamma$.

\begin{theorem}\label{thm:l1-lattice}
Let $M \subseteq N$ be transitive models of $\ZF$ and let $E \in M$ be
a real Banach lattice. If $N \models$ ``$E
\cong \ell_1(\Gamma)$ isometrically as Banach lattices'', then the same
holds in~$M$ for some $\Gamma \in M$.
\end{theorem}

\begin{proof}
Let $Y$ be the set of positive normalised atoms of~$E$:
\[
  y \in Y
\]
iff $y \geqslant 0$, $\|y\| = 1$, and whenever $0 \leqslant z \leqslant
y$ there is a scalar $0 \leqslant \alpha \leqslant 1$ such that
$z = \alpha y$. This definition is made in the language of Banach
lattices and hence gives $Y \in M$.

In~$N$, under an isometric lattice isomorphism $E \cong \ell_1(\Gamma)$,
the set $Y$ is exactly the set of standard unit vectors.

For $x \in E_+$ and $y \in Y$, define
\[
  \alpha_y(x) = \sup \bigl\{ \alpha \in \bbR_+ :
  \alpha y \leqslant x \bigr\}.
\]
In $\ell_1(\Gamma)$ this is precisely the $y$-coordinate of~$x$. For
general $x \in E$, put
\[
  \alpha_y(x) = \alpha_y(x^+) - \alpha_y(x^-).
\]
Now define
\[
  \theta : E \longrightarrow \ell_1(Y), \qquad \theta(x)(y) = \alpha_y(x).
\]
In~$N$, the map $\theta$ is the usual coordinate map from
$\ell_1(\Gamma)$ to $\ell_1(Y)$, and is therefore an isometric lattice
isomorphism onto $\ell_1^N(Y)$. By Lemma~\ref{lem:descent}, and by
old-part absoluteness for $\ell_1(Y)$ under the scalar convention,
$\theta$ is an isometric lattice isomorphism $E \cong \ell_1^M(Y)$
in~$M$.
\end{proof}

\section{\texorpdfstring{$\ZFC$-descent and $\ZF$-obstructions}{ZFC-descent and ZF-obstructions}}\label{sec:zfc}

We now turn to standardness properties that are downward absolute in
$\ZFC$, but not in $\ZF$. The typical situation is that canonical
reconstruction produces a bundle of locally standard pieces — fibres,
signs, phases, automorphism torsors — which $\ZFC$ permits to
trivialise. We begin with the cleanest example, where there are no
completeness or new-sequences caveats.

\subsection{\texorpdfstring{Finite covers: the basic $\ZFC$-only example}{Finite covers: the basic ZFC-only example}}
\label{subsec:finite-covers}

Fix an integer $n \geqslant 2$. Let $\cC_n$ be the class of equivalence
relations isomorphic to the standard equivalence relation on
$Y \times n$, where
\[
  (y, k) \sim (y', k') \quad \Longleftrightarrow \quad y = y'.
\]

\begin{proposition}\label{prop:finite-covers-zfc}
The class $\cC_n$ is downward absolute between transitive models of
$\ZFC$.
\end{proposition}

\begin{proof}
Let $M \subseteq N$ be transitive models of $\ZFC$, and let $(E, \sim)
\in M$. If $N$ sees $(E, \sim) \cong Y \times n$, then every
$\sim$-class has exactly~$n$ elements. This is absolute to~$M$. Let
$Q = E/{\sim}$, which belongs to~$M$. Since $M \models \ZFC$, choose
in~$M$ a bijection from each equivalence class onto $\{0, \ldots, n - 1\}$.
These choices assemble to an isomorphism
\[
  (E, \sim) \cong Q \times n
\]
inside~$M$.
\end{proof}

\begin{proposition}\label{prop:finite-covers-zf-fail}
Let $M \models \ZF$ be transitive and suppose that $M$ contains a
family $(P_i)_{i \in I}$ of two-element sets with no choice function.
If $N \supseteq M$ is an outer transitive model containing a choice
function for this family, then for every fixed $n \geqslant 2$ there is
an equivalence relation $(E, \sim) \in M$ such that
\[
  M \models (E, \sim) \not\cong Y \times n \text{ for every } Y,
\]
but
\[
  N \models (E, \sim) \cong I \times n.
\]
Consequently, relative to the existence of such transitive models, the
class $\cC_n$ is not downward absolute in $\ZF$.
\end{proposition}

\begin{proof}
For each $i \in I$, form an $n$-element set
\[
  Q_i = P_i \times \{0\} \ \cup\ \{(i, k) : 2 \leqslant k < n\}.
\]
The elements $(i, k)$ for $2 \leqslant k < n$ are distinguished
fillers, while the two elements of $P_i \times \{0\}$ remain unordered.

Let
\[
  E = \{(i, q) : i \in I,\ q \in Q_i\},
\]
and put
\[
  (i, q) \sim (j, r) \quad \Longleftrightarrow \quad i = j.
\]
Each equivalence class has exactly~$n$ elements.

If $M$ saw $(E, \sim) \cong Y \times n$, then each class would be
labelled by $\{0, \ldots, n - 1\}$. In the class over~$i$, look at the
two elements coming from $P_i \times \{0\}$ and choose the one whose
label is smaller. This gives a choice function for $(P_i)_{i \in I}$,
contradiction.

On the other hand, any outer transitive model containing a choice
function for $(P_i)_{i \in I}$ can first choose one element from
each~$P_i$ and then label each~$Q_i$ by $\{0, \ldots, n - 1\}$. Hence
such an outer model sees
\[
  (E, \sim) \cong I \times n. \qedhere
\]
\end{proof}

\subsection{\texorpdfstring{Bare $\ell_1(\Gamma)$ as a Banach space}{Bare ell_1(Gamma) as a Banach space}}\label{subsec:l1-bare}

\begin{theorem}\label{thm:l1-zfc}
Let $M \subseteq N$ be transitive models of $\ZFC$ and let $E \in M$ be
a Banach space over the named scalar field~$\bbK$. If $N \models$
``$E \cong \ell_1(\Gamma)$ linearly isometrically'', then the same
holds in~$M$ for some $\Gamma \in M$.
\end{theorem}

\begin{proof}
If $E = \{0\}$, then $M$ already sees $E \cong \ell_1(\emptyset)$, so
there is nothing to prove. Assume henceforth that $E \ne 0$.

The set $D = \Ext(\Ball(E))$ of extreme points of the closed unit ball
is definable from the norm. In $\ell_1(\Gamma)$, the extreme points
are exactly $\{\lambda e_\gamma : \gamma \in \Gamma,\ |\lambda| = 1\}$:
a vector of norm strictly less than~$1$ can be perturbed in one
coordinate, and a norm-one vector with at least two non-zero coordinates can
be perturbed in opposite directions in two of them; while a one-coordinate
unit vector is plainly extreme. Let $\bbT = \{\lambda \in
\bbK : |\lambda| = 1\}$ act on~$D$ by scalar multiplication; the
quotient $Y = D/\bbT$ is, in~$N$, in bijection with~$\Gamma$. Both $D$
and~$Y$ belong to~$M$.

Using $\ZFC$ in~$M$, choose a section $s : Y \to D$. For $x \in E$ and
$y \in Y$, let $\theta(x)(y)$ be the unique scalar~$\alpha$ minimising
\[
  \beta \longmapsto \|x - \beta s(y)\|.
\]
Equivalently,
\[
  \|x - \alpha s(y)\| \leqslant \|x - \beta s(y)\|
  \qquad (\beta \in \bbK).
\]
In~$N$, after identifying $E$ with $\ell_1(\Gamma)$ and $s(y)$ with a
chosen unit vector on the corresponding coordinate line, this unique
minimiser is exactly the corresponding coordinate of~$x$. Hence $N$ sees
$\theta$ as the usual coordinate map
\[
  E \longrightarrow \ell_1^N(Y).
\]
In particular, $N$ sees $\theta$ as linear, isometric and onto. By
Lemma~\ref{lem:descent}, using old-part absoluteness for $\ell_1(Y)$,
the same map is an isometric linear isomorphism
\[
  E \cong \ell_1^M(Y)
\]
in~$M$.
\end{proof}

\begin{remark}[Other $\ell_p$ spaces]
For $1 < p < \infty$, $p \ne 2$, analogous $\ZFC$-descent statements
should follow from the Banach--Lamperti description of the surjective
linear isometries of $\ell_p(\Gamma)$ \cite{FleJam}, after one first
gives an intrinsic definition of the coordinate one-dimensional bands.
We do not need this extension in the present note. The case $p = 2$ is
basis-dependent and is better treated separately as a Hilbert-space
basis question; see Subsection~\ref{subsec:hilbert-basis}.
\end{remark}

\subsection{\texorpdfstring{Bare $c_{00}(\Gamma)$ as a normed space}{Bare c_00(Gamma) as a normed space}}\label{subsec:c00-bare}

Let $c_{00}(\Gamma)$ denote the real or complex vector space of
finitely supported scalar functions on~$\Gamma$, equipped with its
$\ell_1$ norm. In Theorem~\ref{thm:c00-zfc} the scalar field is the
named field~$\bbK$ from Convention~\ref{conv:scalars}; in
Proposition~\ref{prop:c00-sign-torsor} we specialise to the real case.

\begin{theorem}\label{thm:c00-zfc}
Let $M \subseteq N$ be transitive models of $\ZFC$, and let $E \in M$
be a normed space over the named scalar field~$\bbK$. If
\[
  N \models \text{``$E \cong c_{00}(\Gamma)$ linearly isometrically
  for some set $\Gamma$''},
\]
then
\[
  M \models \text{``$E \cong c_{00}(Y)$ linearly isometrically for some
  set $Y$''}.
\]
\end{theorem}

\begin{proof}
The zero space is trivial, so assume $E \ne 0$. Let
\[
  D = \Ext(\Ball(E)).
\]
In $c_{00}(\Gamma)$ with the $\ell_1$ norm, the extreme points of the
closed unit ball are exactly
\[
  \{\lambda e_\gamma : \gamma \in \Gamma,\ |\lambda| = 1\}.
\]
The verification is the same elementary perturbation argument as for
$\ell_1(\Gamma)$ above.
Let $\bbT = \{\lambda \in \bbK : |\lambda| = 1\}$ and put
$Y = D/\bbT$. Using $\ZFC$ in~$M$, choose a section $s : Y \to D$.

For $x \in E$ and $y \in Y$, let $\theta(x)(y)$ be the unique
scalar~$\alpha$ minimising
\[
  \beta \longmapsto \|x - \beta s(y)\|.
\]
In~$N$, this is exactly the $y$-coordinate of~$x$. Hence $N$ sees
\[
  \theta : E \to c_{00}^N(Y)
\]
as the standard coordinate isometry. Since old finitely supported
families on~$Y$ are exactly the old elements of $c_{00}^N(Y)$, the
descent lemma yields
\[
  M \models E \cong c_{00}^M(Y). \qedhere
\]
\end{proof}

\subsection{\texorpdfstring{$\ZF$-obstruction: a sign torsor}{ZF-obstruction: a sign torsor}}\label{subsec:sign-torsor}

The use of $\ZFC$ in Theorems~\ref{thm:l1-zfc} and~\ref{thm:c00-zfc} is
essential. The cleanest unconditional torsor obstruction is given by
the finite-support analogue $c_{00}(I)$, which has no completeness or
new-sequence caveat. We treat both forms.

The following obstruction is stated for real normed spaces. A complex
analogue can be formulated using circle torsors, but the real
two-point version is the cleanest and is sufficient for the
$\ZF$-failure phenomenon.

\begin{proposition}[Sign torsors for $\ell_1$]\label{prop:l1-zf-failure}
Let $M \models \ZF$ be transitive, and suppose that $M$ contains a
family $(P_i)_{i \in I}$ of two-element sets with no choice function.
For each~$i$, let
\[
  L_i = \Bigl\{ f : P_i \to \bbR :
  \sum_{p \in P_i} f(p) = 0 \Bigr\},
\]
and equip $L_i$ with the norm
\[
  \|f\| = \frac{1}{2} \sum_{p \in P_i} |f(p)|.
\]
Thus $L_i$ is a one-dimensional real normed space, and its unit sphere
is canonically identified with~$P_i$: the point $p \in P_i$
corresponds to the function $u_p$ satisfying
\[
  u_p(p) = 1, \qquad u_p(q) = -1
\]
for the other element $q \in P_i$.

Form
\[
  E = \bigoplus_{i \in I}^{\ell_1} L_i.
\]
Then $M$ does not contain a linear isometry
\[
  E \cong \ell_1(J)
\]
for any set~$J$.

If $N \supseteq M$ is an outer transitive model with the same scalar
field, containing a choice function for $(P_i)_{i \in I}$, and
satisfying the no-new-$\ell_1$-vectors conditions
\[
  \ell_1^N(I) = \ell_1^M(I)
\]
and
\[
  \left( \bigoplus_{i \in I}^{\ell_1} L_i \right)^N
  = \left( \bigoplus_{i \in I}^{\ell_1} L_i \right)^M,
\]
then
\[
  N \models \text{``$E \cong \ell_1(I)$ linearly isometrically''}.
\]
\end{proposition}

\begin{proof}
The extreme points of the closed unit ball of~$E$ are exactly the unit
vectors in the summands~$L_i$. Hence an isometry $E \cong \ell_1(J)$
would send these extreme points onto $\{\pm e_j : j \in J\}$ and would
choose, for each~$i$, the unique unit vector in~$L_i$ mapped to a
positive basis vector. Since the unit sphere of~$L_i$ is canonically
$P_i$, this gives a choice function for $(P_i)_{i \in I}$,
contradiction.

Conversely, a choice function for $(P_i)$ chooses a unit vector
$u_i \in L_i$ for every~$i$. The map
\[
  \sum_i \alpha_i u_i \longmapsto (\alpha_i)_{i \in I}
\]
is then the desired isometry, because the no-new-vectors hypotheses
ensure that the domain and target computed in~$N$ are exactly the old
ones.
\end{proof}

\begin{proposition}[Finite-support version]\label{prop:c00-sign-torsor}
With $(P_i)$ and $L_i$ as above, put
\[
  E_{00} = \bigoplus_{i \in I}^{00} L_i,
\]
the algebraic finite-support direct sum with its $\ell_1$ norm. Then
$M$ does not see $E_{00}$ as linearly isometric to $c_{00}(J)$ for any
set~$J$, but every outer transitive model with the same scalar field
and containing a choice function for $(P_i)$ sees
\[
  E_{00} \cong c_{00}(I).
\]
\end{proposition}

\begin{proof}
The obstruction in~$M$ is the same as in
Proposition~\ref{prop:l1-zf-failure}: an isometry with $c_{00}(J)$
would choose a positive unit vector in each summand~$L_i$, hence a
choice function for $(P_i)$.

Conversely, a choice function for $(P_i)$ chooses unit vectors
$u_i \in L_i$, and the finite-support map
\[
  \sum_{i \in F} \alpha_i u_i \longmapsto
  \sum_{i \in F} \alpha_i e_i
\]
is a linear isometry $E_{00} \cong c_{00}(I)$. There is no issue about
new supports: finite subsets of ground-model sets are already
ground-model sets. Under the scalar convention, there are no new scalar
coefficients either.
\end{proof}

\subsection{\texorpdfstring{Hilbert spaces as $\ell_2(\Gamma)$}{Hilbert spaces as ell_2(Gamma)}}\label{subsec:hilbert-basis}

The operator-algebra statement $A \cong \cB(H)$ is basis-free, but the
standard presentation $H \cong \ell_2(\Gamma)$ is not.

\begin{proposition}\label{prop:hilbert-zfc}
Between transitive models of $\ZFC$, the property
\[
  H \cong \ell_2(\Gamma) \text{ for some set } \Gamma
\]
is downward absolute for Hilbert spaces $H$ over the named scalar field.
\end{proposition}

\begin{proof}
If $M \models \ZFC$ and $H \in M$ is a Hilbert space, then $M$ contains
an orthonormal basis of~$H$ by the usual Zorn-lemma argument. Hence
$M$ already sees $H \cong \ell_2(\Gamma)$ for some~$\Gamma$.
\end{proof}

\begin{remark}
In $\ZF$, the assertion that every Hilbert space has an orthonormal
basis is a genuine choice principle. Thus a Hilbert space without an
orthonormal basis in a ground model may become isomorphic to
$\ell_2(\Gamma)$ in an outer model which adds such a basis and does not
create additional square-summable coordinate families. This is the
Hilbert-space analogue of the $\ell_1$ sign-torsor caveat.
\end{remark}

\section{\texorpdfstring{Consequences for $\Pi^1_1$-definability}{Consequences for Pi^1_1-definability}}\label{sec:pi11}

The preceding results have two different kinds of consequences. For
the canonically reconstructible classes, they show that the usual
forcing strategy for disproving $\Pi^1_1$-definability cannot work. In
the motivating case of full symmetric groups,
Corollary~\ref{cor:sym-pi11} gives the stronger positive conclusion
that fullness itself has a uniform $\Pi^1_1$ definition over transitive
$\ZF$-models. For the torsor examples, by contrast, one obtains actual
failures of $\Pi^1_1$-definability over~$\ZF$.

\begin{corollary}\label{cor:no-forcing-counterexample-zf}
Let $\cC$ be any of the $\ZF$-descent classes proved above; for example
\[
  \Sym(X), \quad X^X, \quad \cP(X), \quad \Rel(X), \quad \mathscr{O}_X,
  \quad \Pi(X), \quad R^X
\]
with $R$ finitely generated and centrally indecomposable,
\[
  \ell_\infty(X), \quad c_0(X), \quad \cB(H), \quad \cK(H),
\]
or $\ell_1(X)$ as a real Banach lattice. Then there are no transitive models
$M \subseteq N$ of $\ZF$ and no $A \in M$ such that
\[
  M \models A \notin \cC \qquad \text{but} \qquad N \models A \in \cC.
\]
In particular, no forcing extension of a $\ZFC$ ground model can turn a
non-standard object of one of these kinds into a standard one.
\end{corollary}

\begin{corollary}\label{cor:no-forcing-counterexample-zfc}
For the $\ZFC$-descent classes — bare Banach-space isometry with some
$\ell_1(\Gamma)$, finite covers $Y \times n$, normed-space isometry
with some $c_{00}(\Gamma)$, and Hilbert-space isomorphism with
$\ell_2(\Gamma)$ — there are no transitive models $M \subseteq N$
of~$\ZFC$ and no $A \in M$ such that
\[
  M \models A \notin \cC \quad \text{but} \quad N \models A \in \cC.
\]
Thus the original forcing strategy cannot refute a putative
$\Pi^1_1$ definition over transitive $\ZFC$-models.
\end{corollary}

The next two corollaries are to be read relative to the standard
existence, obtained for example by symmetric-model methods, of
transitive $\ZF$-models containing a family of pairs with no choice
function and an outer transitive model adding such a choice function
and, for the normed-space statement, preserving the named scalar field;
cf.~\cite{Jech-AC,Howard-Rubin}.

\begin{corollary}\label{cor:zf-no-pi11-finite-covers}
Assume that there are transitive models $M \subseteq N$ of $\ZF$ such
that $M$ contains a family of two-element sets with no choice function
and $N$ contains a choice function for that family. Then, for each
$n \geqslant 2$, the class of equivalence relations isomorphic to
$Y \times n$ has no $\Pi^1_1$ definition uniform over transitive
$\ZF$-models with absolute matrix.
\end{corollary}

\begin{proof}
By Proposition~\ref{prop:finite-covers-zf-fail}, this class is not
downward absolute between such transitive $\ZF$-models. But any uniform
$\Pi^1_1$ definition with absolute matrix would be downward absolute by
Lemma~\ref{lem:pi11-test}.
\end{proof}

\begin{corollary}\label{cor:zf-no-pi11-c00}
Under the same transitive-model hypothesis, the property, for real
normed spaces, ``is linearly isometric to $c_{00}(J)$ for some set
$J$'' has no $\Pi^1_1$ definition uniform over transitive $\ZF$-models
with absolute matrix.
\end{corollary}

\begin{proof}
Use Proposition~\ref{prop:c00-sign-torsor} and
Lemma~\ref{lem:pi11-test}.
\end{proof}

\begin{remark}
For the Banach-space property ``is linearly isometric to some
$\ell_1(\Gamma)$'', the same $\ZF$-level non-definability conclusion is
available whenever the sign-torsor construction is carried out inside
outer models preserving the relevant $\ell_1$-summable families.
Without that preservation hypothesis, the finite-cover and
finite-support versions are the clean unconditional examples.
\end{remark}

\section{Concluding remarks and questions}\label{sec:questions}

The examples collected above suggest a general pattern and several
natural open problems.

\begin{question}
Is there a syntactic criterion on a functor $F : \mathbf{Set} \to
\mathbf{Struct}$ equivalent to downward absoluteness, in $\ZF$, of the
class of structures of the form $F(X)$? A candidate condition is that
the isomorphism class of $F(X)$ has a \emph{single-sorted definable
skeleton} (in the sense of Rubin~\cite{Rubin}) reconstructing~$X$.
\end{question}

\begin{question}
Which classes of Banach algebras satisfy $\ZFC$-downward absoluteness
of standardness? For instance, is the property ``is isometrically
isomorphic to the group algebra $L^1(G)$ for some locally compact
abelian group~$G$'' downward absolute under forcing preserving the
reals?
\end{question}

\begin{remark}
The automorphism-tower phenomenon of Fuchs and
Hamkins~\cite{FuchsHamkins} shows that certain isomorphism relations
between algebraic objects are highly forcing-controllable.
Theorem~\ref{thm:symmetric-groups} shows that the standardness relation
`is a full symmetric group' is rigid in the opposite direction:
forcing cannot turn a ground-model non-full group into a full symmetric
group.
\end{remark}

\section*{Acknowledgements}
The author gratefully acknowledges support received from
NCN Sonata-Bis~13 (2023/50/E/\-ST1/00067).

\end{document}